\documentclass{amsart}

\usepackage{hyperref}
\usepackage[alphabetic, msc-links]{amsrefs}

\usepackage{amsthm,amssymb,verbatim}
\usepackage{graphicx}
\usepackage{enumerate}
\usepackage{booktabs}
\usepackage{color}

\newcommand{\MM}{\mathcal M}

 \newcommand{\OO}{\mathbf{O}} 

 \newcommand{\RR}{\mathbf{R}}  
 \newcommand{\BB}{\mathbf{B}}  
 \newcommand{\Ss}{\mathbf{S}}  

 \newcommand{\area}{\operatorname{area}}
 \newcommand{\eps}{\epsilon}

    \newtheorem{theorem}    {Theorem}     
    \newtheorem*{theorem1'}{Theorem 1$'$} 
    \newtheorem*{theorem1*}{Theorem 1$^*$}
    \newtheorem*{no-theorem}{}
    
    \newtheorem*{corollary}{Corollary}

    \newtheorem{claim}  {Claim}

    \theoremstyle{definition}

    \theoremstyle{definition}

\usepackage{hyperref}   
\usepackage[alphabetic]{amsrefs}

\begin{document}

  \renewcommand{\marginpar}[1]{}

\renewcommand{\thesubsection}{\thetheorem}

\title[Sharp Lower Bounds on Density of Area-Minimizing Cones]{Sharp Lower Bounds on Density \\ for Area-Minimizing Cones}
\author{Tom Ilmanen}
\address{Departement Mathematik ETH Zentrum, R\"amistrasse 101 CH-8092 Z\"urich, Switzerland}
\email{ilmanen@math.ethz.ch}
\author{Brian White}
\address{Department of Mathematics\\ Stanford University\\ Stanford, CA 94305, USA}
\email{bcwhite@stanford.edu}
\thanks{The research of the second author was supported by the National Science Foundation 
  grants~DMS-0707126 and DMS1105330.}      
\begin{date}{July 23, 2010. Revised May 12, 2013.}\end{date}

\begin{abstract}
We prove that the density of a topologically nontrivial, area-minimizing hypercone
with an isolated singularity
must be greater than $\sqrt2$.  The Simons' cones show that $\sqrt2$
is the best possible constant.   
If one of the components of the complement of the cone has nontrivial $k^{\rm th}$ homotopy
group, we prove a better bound 
in terms of $k$; that bound is also best possible.  The proofs use mean curvature flow.
\end{abstract}

\maketitle

In this paper, we prove some sharp lower  bounds on densities
of area-minimizing hypercones or, equivalently, on volumes of certain closed
minimal hypersurfaces in round spheres.  
We begin by indicating why such
density bounds are of interest.
Recall that if $M$ is an $m$-dimensional minimal variety in a Riemannian manifold
and if $x$ is an interior point of $M$, then the density of $M$ at $x$
is
\begin{equation}\label{density}
   \Theta(M,x) := \lim_{r\to 0} \frac{\area(M\cap \BB(x,r)}{\omega_mr^m},
\end{equation}
where $\omega_m$ is the $m$-dimensional volume of the unit ball in $\RR^m$.
The limit exists by the monotonicity formula.  
The density is $1$ at any multiplicity $1$ regular point, and it is strictly greater than
$1$ at any singular point (by Allard's regularity theorem).
If $M$ is a cone with vertex $x$, then
the ratio in~\eqref{density} is independent of $r$; in that case, we write $\Theta(M)=\Theta(M,x)$.

This paper treats the case of area-minimizing hypersurfaces (either integral currents
or flat chains mod $2$).   Consider the following question:

\begin{quote}
{\bf Q1}.
What is the infimum of $\Theta(M,x)$
among all pairs $(M,x)$ where $M$ is an area minimizing hypersurface (in some Riemannian
manifold) and $x$ is an interior singular point of $M$?
\end{quote}

Here ``interior point of $M$'' means ``point in the support of $M$ but not in the support of $\partial M$".

Note that if $x$ is an interior singular point of $M$ and if $C$
is a tangent cone to $M$ at $x$, then $C$ is an area minimizing hypercone in Euclidean
space with a singularity at its vertex, and $\Theta(C)=\Theta(M,x)$.  Furthermore,
standard dimension reducing arguments show that either $C$ has an isolated singularity at its vertex,
or else there is another area minimizing hypercone $C'$ of lower dimension such that
$C'$ has an isolated singularity at vertex and such that $\Theta(C') \le \Theta(C)$.
Thus the question Q1 is equivalent to:

\begin{quote}
{\bf Q2}.
What is the infimum of $\Theta(C)$ among all area-minimizing hypercones $C$ such that
$C$ has an isolated singularity at the origin?
\end{quote}

In this paper, we give a sharp answer to question Q2 provided one restricts the cones $C$
to those that are topologically nontrivial.  In particular, we prove:

\begin{theorem}\label{MainTheorem}
Suppose that $C\subset \RR^n$ is an area-minimizing hypercone with an isolated singularity at the origin.
Suppose also that $C$ is topologically nontrivial in the following sense: at least one of the two components of $\RR^n\setminus C$ is non-contractible.   
Then the density of $C$ at the origin is greater than $\sqrt2$.
\end{theorem}

If one wants a constant independent of the dimension of the dimension $n$, then
$\sqrt2$ is the best possible because the Simons' cone
\[
  C_{m,m} := \{ (x,y)\in \RR^m\times \RR^m =\RR^{2m}:  |x|=|y| \}
\]
is both topologically nontrivial and area-minimizing for $m\ge 4$,
and by a straightforward calculation (see section~\ref{Examples})
its density $\Theta(C_{m,m})$ tends to $\sqrt2$ as $m$ tends to infinity.

Many years ago Bruce Solomon conjectured that the answer to question Q2 above is $\sqrt2$.
Our work shows that Solomon's conjecture is true in the class of topologically nontrivial cones.

We remark that all of the many known examples
 (see~\cite{Lawlor-Criterion}) of area-minimizing hypercones $C$
with isolated singularities are topologically nontrivial.   On the other hand, there are 
examples 
(\cite{Hsiang-I}, \cite{Hsiang-II}, \cite{Hsiang-III})
of minimal embeddings of $m$-spheres into $(m+1)$-spheres that are not totally geodesic.
The corresponding cones are minimal hypercones with isolated singularities at their vertices
and are topological trivial.  However, we not know whether those cones minimize area.


Theorem~\ref{MainTheorem} can be restated in terms of minimal submanifolds of spheres:

\marginpar{Maybe this should be Theorem 1'}
\begin{theorem1'}
Let $\Gamma$ be a closed minimal hypersurface in the unit sphere $\partial \BB\subset \RR^n$.  Suppose that the corresponding
cone
\[
   C = \{ rx: x\in \Gamma, r\ge 0\}
\]
is area-minimizing.  Suppose also that at least one of the components of $(\partial \BB)\setminus \Gamma$ is not
contractible.  Then the area of $\Gamma$ is greater than $\sqrt2$ times the area of the totally geodesic $(n-2)$-sphere
in $\partial \BB$.
\end{theorem1'}

To see that the topological hypotheses of the two theorems are equivalent, 
note that if $U$ is a component of $\RR^n\setminus C$, then $U\cap \partial \BB$ is the corresponding component
of $(\partial \BB)\setminus \Gamma$.  Also, $U$ is homeomorphic to $(U\cap\partial \BB)\times \RR$, so
$U$ is contractible if and only if $U\cap \partial \BB$ is contractible.

Let $C$ be a cone as in Theorem~\ref{MainTheorem}.
Since one of the components of $\RR^n\setminus C$ is non-contractible, one of its homotopy groups,
say the $k^{\rm th}$ homotopy group, is nontrivial.    One can get a better lower bound for $\Theta(C)$
if one allows a constant that depends on $k$.  In particular,
we show in Theorem~\ref{TopTheorem} that
\[
  \Theta(C) \ge d_k = \left( \frac{k}{2\pi e} \right)^{k/2}\sigma_k
\]
where $d_k$ is the Gaussian density of a shrinking $k$-dimensional
sphere and $\sigma_k$ is the area of the unit $k$-dimensional sphere.  (Gaussian density
plays the role in mean curvature flow that density does is minimal surface theory.
See~\cite{White-Stratification}.)

As before, this result is sharp:  for any $\eps>0$, there is an $n$ and a cone 
$C\subset \RR^n$ such that $C$ satisfies the hypotheses of the theorem and such
that
\[
  \Theta(C) < d_k + \eps.
\]
See section~\ref{Examples}.

There are no singular area-minimizing hypercones of dimension less than $7$.
In another paper we will show that Theorem~\ref{MainTheorem} continues
to hold if the the hypothesis that the cone is area-minimizing is replaced by the hypothesis
that the cone has dimension less than $7$.  We will also prove a somewhat weakened version
of Theorem~\ref{TopTheorem} for cones of dimension less than $7$.

The proofs of the theorems use various facts about mean curvature flow.
We describe those facts as we need them.  For readers who would prefer to know
the relevant facts ahead of time, we have stated most of them in 
appendix~\ref{appendix}.  However, the proof of theorem~\ref{MainTheorem} 
requires much less than is stated
there.  In particular,
for theorem~\ref{MainTheorem}, 
it suffices to consider mean curvature flow up to the first singular time.

\section{The proof of Theorem~\ref{MainTheorem}}\label{Proof-One-Section}

\begin{proof}[Proof of theorem~\ref{MainTheorem}]
Let $M$ be any minimal variety in $\RR^n$.  Then $M$ is in equilibrium
for mean curvature flow. That is, 
\begin{equation}\label{static}
  \MM: t\in \RR \mapsto M
\end{equation}
is a mean curvature flow.   Note that $(x,t)$ is a singular point of $\MM$
if and only if $x$ is a singular point of $M$.  Furthermore, the Gaussian
density of $\MM$ at $(x,t)$ is equal to the density of $M$ at $x$:
\[
  \Theta(\MM, (x,t)) = \Theta(M,x).
\]
(The definition and basic properties of Gaussian density may be found
in~\cite{White-Stratification}, for example.)

Recall the following upper semicontinuity property of Gaussian density (which follows
fairly easily from Huisken's monotonicity of density):  if $\MM_i$ is a sequence
of mean curvature flows that converge (as Brakke flows) to $\MM$ and if $X_i$
is a sequence of points in spacetime that converge to $X$, then
\[
   \Theta(\MM, X) \ge \limsup \Theta(\MM_i, X_i).
\]

Now let $M$ be the portion of the cone $C$ in the unit ball $\BB$ centered at the origin: $M=C\cap \BB$.
As above, we let $\MM$ be the static mean curvature flow~\eqref{static}.
The singular points of $\MM$ are precisely the points $(0,t)$, $t\in \RR$.

We will construct for each sufficiently small $\eps>0$ a mean curvature flow
\[
   \MM_\eps: t\in \RR \mapsto M_\eps(t)
\]
with the following properties:
\begin{enumerate}[\upshape(p1)]
\item\label{first} $\MM_\eps\to \MM$ as $\eps \to 0$.
\item $\MM_\eps$ has a singular point $X_\eps=(x_\eps,0)$.
\item $\MM_\eps$ has as a tangent flow at $X_\eps$ a self-similarly
shrinking cylinder $\partial \Ss^k \times \RR^{n-k-1}$ for some $k=k_\eps \le n-2$.
\item $X_\eps\to (0,0)$ as $\eps\to 0$.
\end{enumerate}

We now explain why existence of such flows $\MM_\eps$ implies the theorem.
Let $d_k$ be the Gaussian density of a self-similarly shrinking cylinder
$\Ss^k\times \RR^{n-k-1}$.  (Here $\Ss^k=\partial \BB^{k+1}$ is the unit $k$-sphere in 
$\RR^{k+1}$.)  This Gaussian density is independent of $n$, and
hence is also the Gaussian density of a shrinking $k$-sphere in $\RR^{k+1}$.   Indeed,
\begin{equation}\label{GaussianDensityValues}
\begin{aligned}
    d_k 
    &= \left(\frac{k}{2\pi e}\right)^{k/2}\sigma_k  \\
    &= 2\sqrt{\pi} \left( \frac{k}{2e}\right)^{k/2} \left( \frac1{\Gamma(\frac{k+1}2)} \right)
\end{aligned}
\end{equation}
(where $\sigma_k$ is the area of a $k$-dimensional sphere of radius $1$),
from which it follows that 
\begin{equation}\label{dn-limit}
 \text{$d_1 > d_2 > d_3 >\dots$ and  $\lim_{k\to\infty}d_k = \sqrt2$}.
\end{equation}
(See~\cite{Stone-Density}*{Appendix A} for proofs of these facts about $d_k$.)
Thus
\[
  \Theta(\MM_\eps, X_\eps) = d_{k(\eps)} \ge d_{n-2},
\]
so
\[
  \Theta(C,0)
  =
  \Theta(\MM_0, (0,0))
  \ge
  \limsup_{\eps\to 0} \Theta(\MM_\eps, X_\eps) 
  \ge d_{n-2} > \sqrt2,
\]
as desired.  Hence to prove the theorem, it suffices to construct mean curvature flows
$\MM_\eps$ with properties (p1) - (p4).

Since $C$ is an area-minimizing hypercone, it is one leaf of a foliation of $\RR^N$
by area-minimizing hypersurfaces~\cite{Hardt-Simon-Cone}.  
The foliation is of course singular at the origin,
but it is smooth everywhere else.    
Each leaf other than $C$ is a radial graph over one of the components of 
  $(\partial \BB)\setminus C$.
Furthermore, the foliation is invariant under positive dilations about the origin. 
We can conveniently represent the leaves as level sets of a function $u: \RR^n\to \RR$
as follows.
Let $U^+$ and $U^-$ be the two connected components of $\RR^N\setminus C$.
By hypothesis, at least one of $U^+$ and $U^-$ is non-contractible.
We choose the labeling so that $U^-$ is non-contractible.
We let $u(x)=0$ for points $x\in C$.   For points $x\in \RR^n\setminus C$,
we let $u(x)$ be plus or minus the distance from the origin to the leaf through $x$ according to whether
$x\in U^+$ or $x\in U^-$.  Note that
\begin{equation}\label{homogeneous}
 \text{$u(\lambda x) = \lambda u(x)$ for $\lambda>0$}
\end{equation}
since the foliation is invariant under positive dilations.

Let 
\[
  L_v := u^{-1}(v) \cap \BB
\]
be the portion of the leaf $u^{-1}(v)$ inside the unit ball $\BB$.

Since the leaf $u^{-1}(v)$ converges smoothly (away from the origin) to $C$ as $v\to 0$, it follows that there is a $\delta>0$
such that the leaf $u^{-1}(v)$ intersects the unit sphere $\partial \BB$ transversely provided $|v|\le \delta$.   We now construct the mean curvature
flow $\MM_\eps$ for each $\eps$ with $0<\eps\le \delta$.

Our mean curvature flow $\MM_\eps$ will be a flow of a surface with boundary.  We first describe the motion of the boundary.
Let $\tau: \RR\to [-1,1)$ be a smooth increasing function such that 
\begin{enumerate}
\item $\tau(t)=-1$ for $t\le 0$,
\item $\tau'(t)>0$ for $t>0$, and
\item $\tau(t)\to 1$ as $t\to\infty$.
\end{enumerate}
For $t \in \RR$, let
\[
    \Gamma_\eps(t) =  \partial L_{\eps\tau(t)}.
\]
Thus for $t\le 0$, the boundary $\Gamma_\eps(t)$ is $\partial L_{-\eps}$.  As $t\to\infty$, 
$\Gamma_\eps(t)$
converges smoothly to $\partial L_\eps$.   Note that $t\mapsto \Gamma(t)$ is a smooth isotopy in $\partial \BB$.

Now let
\[
   \MM_\eps: t\in [0,\infty) \mapsto M_\eps(t)
\]
be the mean curvature flow (of surfaces with boundary) such that:
\begin{enumerate}
\item
Initially (i.e., at time $t=0$), the surface is $L_{-\eps}$.
\item At each time $t\ge 0$, the boundary of the surface is $\Gamma_\eps(t)$.
\end{enumerate}
(For existence of the flow, 
see theorem~\ref{BoundaryExistenceTheorem} in the appendix, which is proved by elliptic regularization as in~\cite{Ilmanen-Elliptic}. 
However, for the proof of theorem~\ref{MainTheorem}, 
it would suffice to consider the 
flow $t\in [0,T)\mapsto M_\eps(t)$ up to the first singular time.) 

Note that since the surface is initially minimal, and
since the boundary is always moving to one side of the surface, it follows (by the strong
maximum principle)
that for $t>0$, the surface $M_\eps(t)$ has
nowhere vanishing mean curvature.  
(See assertion~\ref{Each-M-assertion}
 in theorem~\ref{BoundaryExistenceTheorem} in the appendix.)

We extend the flow to all time by setting 
\[
   M_\eps(t) = M_\eps(0) = L_{-\eps} \quad \text{for $t\le 0$}.
\]

By the maximum principle (or by construction, if one uses elliptic regularization to construct
the flow), the surfaces $M_\eps(t)$ all lie in the region where $-\eps \le u(x)\le \eps$:
\begin{equation}\label{squeeze}
    M_\eps(t) \subset \BB \cap \{x: -\eps \le u(x) \le \eps\}.
\end{equation}
(To see this using the maximum principle, note that the quantity
\[
    \sup \{ |x|: x\in M_\eps(t)\}
\]
is a decreasing function of $t$, that
\[
   \inf \{ u(x): x\in M_\eps(t)\}
\]
is an increasing function of $t$, and that the quantity
\[
  \sup \{ u(x): x\in M_\eps(t)\}
\]
is a decreasing function of $t$ on any interval on which it is $>\eps$.)

As $\eps\to 0$, any subsequence of the flows $\MM_\eps$ has a further subsequence
that converges to a Brakke flow $\MM_0$.   By~\eqref{squeeze}, for each time $t$, $M_0(t)$
is a varifold supported in $C\cap \BB$.  In fact, for $t\le 0$, $M_0(t)$ is the multiplicity
$1$ varifold associated to $C\cap \BB$.   Consequently, either this holds for all $t$,
or else there is some time $T$ such that the $M_0(t)$ vanishes at time $T$. (That is,
such that $M_\eps(t)$ is the zero varifold for $t>T$.)  But such vanishing is impossible because
at all times, $M_0(t)$ supports an integral current whose boundary is an integral current
with support  $\Gamma_0$.

By the local regularity theory in~\cite{White-Local}, 
for all sufficiently small $\eps$, the flows $\MM_\eps$ are regular away
from the line $0\times \RR$.   That is, if $\eps(i)\to 0$ and if $(x_{\eps(i)}, t_{\eps(i)})$ is a singular point
of the flow $\MM_{\eps(i)}$, then $x_i\to 0$.  In particular, there are no singularities at the boundary.

\begin{claim}\label{large-t}  For sufficiently large $t$ (depending on $\eps$) 
the surfaces $M_\eps(t)$ are smooth and converge
smoothly to $L_\eps$.
\end{claim}

It suffices to prove weak convergence to $L_\eps$, since smooth convergence then follows from the local regularity theory~\cite{White-Local}.
There are many ways to prove weak convergence. 
For example, by assertion~\ref{infinity-assertion} of 
theorem~\ref{BoundaryExistenceTheorem} in the appendix, the surface $M_\eps(t)$ converges
as $t\to\infty$ to a minimal hypersurface $H\subset \BB$ with boundary $\partial L_\eps$ and 
with a singular set $Z$ of 
Hausdorff dimension at most $n-8$.  The convergence is smooth away from $Z$.
If the claim were not true, then $v:=\min\{u(x): x\in H\}$ would be strictly less than $\eps$.
Then for $s<v$ and sufficiently close to $v$, the shortest distance from the leaf $L_s$ to $H$
is realized by points $p$ and $q$ where $p$ is in the interior of $L_s$ and $q$ is in $H\setminus \partial H$. 
 It follows that immediately that the tangent cones to $L_s$ at $p$ and to $H$ at $q$ lie in halfspaces and are
therefore planes.  Consequently, $p$ and $q$ are regular points of $L_s$ and of $H$, and therefore
we get a contradiction to the strong maximum principle.) 
This completes the proof of claim~\ref{large-t}.

\begin{claim}\label{singular}
If $\eps$ is sufficiently small, the flow $\MM_\eps$ must have a singularity.
\end{claim}

For suppose not.   
By hypothesis,  the component $U^-$  of $\RR^N\setminus C$ is not contractible. 
Thus at least one of its homotopy groups, say the $k^{\rm th}$, is 
nontrivial.
Hence there is a map $f: \Ss^k\to U^-$ from the $k$-sphere to $U^-$  that is homotopically nontrivial in $U^-$.
By dilating, we may assume that $f(\Ss^k)$ lies in $U^-\cap \BB$.
We assume that $\eps$ is small enough that  $u<-\eps$ on
$f(\Ss^k)$. (In other words, $f(\Ss^k)$ and the origin lie on opposite sides of 
$L_{-\eps}$.)

On the other hand, $f$ is homotopically trivial in $\BB\cap \overline{U^-}$
by the homotopy
\begin{align*}
  &H: \Ss^k\times[0,1]\to \BB\cap\{u\le a\}, \\
  &H(x, v) =  v f(x).
\end{align*}  

Now if the flow $\MM_\eps$ had no singularities, then it would provide a smooth isotopy
from $L_{-\eps}$ to $L_\eps$.   Using the isotopy, we could push 
\[
  H(\Ss^k\times[0,1])
\] 
into $U^-$, leaving $f(\Ss^k)=H(\Ss^k\times\{1\})$ fixed, which means that $f$ would be homotopically trivial in $U^-$,
contradicting the choice of $f$.  The contradiction proves that the flow $\MM_\eps$
has a singularity. 

\begin{claim}\label{neck-pinch}
If $\eps$ is sufficiently small, the flow $\MM_\eps$ must have a singularity 
with a self-similarly shrinking $\Ss^j\times\RR^{n-j}$ as a tangent flow.
\end{claim}

 Let $X_\eps=(x_\eps,t_\eps)$ be a singularity of the flow.  
 For this theorem, we may as well
choose $t_\eps$ to be the first time at which a singularity occurs.  
(In the proof of Theorem~\ref{TopTheorem} below, 
we will make a different choice.) 
Let $\mathcal{T}_\eps$ be a tangent flow to $\MM_\eps$ at $(x_\eps,t_\eps)$. 
Then $\mathcal{T}_\eps$
is a self-similarly shrinking cylinder 
      $\Ss^j \times \RR^{n-j-1}$ for some $j$ with $1\le j \le n-1$
      by theorem~\ref{ConvexTypeTheorem} in the 
      appendix\footnote{By claim~\ref{large-t}, the
      hypothesis of theorem~\ref{ConvexTypeTheorem} holds.}.
      (If $t_\eps$ is the first singular time, this also follows 
      from~\cite{White-Nature}*{Theorem~1}.)
      
This completes the proof of claim~\ref{neck-pinch}, but we remark  that in fact $j\ne n-1$, 
and thus that $j\le n-2$. 
To see this, note that for $t<t_\eps$, the surfaces $M_\eps(t)$
are diffeomorphic to $M_\eps(0)$ and hence are connected manifolds with nonempty boundary.   
Now if $j$ were equal to $n-1$,
that is, if a tangent flow $\mathcal{T}_\eps$ at $(x_\eps,t_\eps)$ were a shrinking sphere, then just before the singularity, $M_\eps(t)$ would have
a compact component diffeomorphic to a sphere,  a contradiction.
  Thus $j\le n-2$.

Now we have proved that $\MM_\eps$ has all the desired properties, except that we do not know that 
$(x_\eps, t_\eps)\to (0,0)$.   Thus we modify the flow by translating in time by $-t_\eps$.   The
new flow $\MM'_\eps$ has a shrinking-cylinder type singularity at $X_\eps' = (x_\eps, 0)$, and $X_\eps'\to (0,0)$.
Furthermore, as $\eps\to 0$, $\MM'_\eps$ converges to $\MM$ for the same reason that $\MM_\eps$ converges to $\MM$.
We have proved that the flow $\MM'_\eps$ has all the properties (p1) - (p4), completing the proof of the theorem.
\end{proof}

We have actually proved a little more than was asserted in the theorem:

\begin{theorem1*}\label{ALittleMore}
Let $C\in \RR^n$ be an area-minimizing hypercone with an isolated singularity at the origin, and suppose that $C$
is topologically nontrivial (as in Theorem~\ref{MainTheorem}).   Then the density of $C$ at the origin is greater than or equal to $d_{n-2}$, the Gaussian
density of a shrinking $(n-2)$-sphere in $\RR^{n-1}$.
\end{theorem1*}

For a better bound, see the corollary to theorem~\ref{TopTheorem} below.

\section{Bounding Density in Terms of Topology}

\begin{theorem}\label{TopTheorem}
Suppose that $C\subset \RR^n$ is an area-minimizing hypercone with an isolated singularity at the origin.
Suppose also that one of the components of $\RR^n\setminus C$ has nontrivial $k^{\rm th}$ homotopy group.
Then
\[
  \Theta(C) \ge d_k
\]
where 
\[
    d_k = \left(\frac{k}{2\pi e}\right)^{k/2}\sigma_k 
\]
is the Gaussian density of a shrinking $k$-sphere in $\RR^{k+1}$.
\end{theorem}

Here $\sigma_k$ is the area of the unit sphere in $\RR^{k+1}$.

\begin{proof}
For $0<\eps\le \delta$, let $\MM_\eps$ be the flow constructed in the proof of 
theorem~\ref{MainTheorem}.
We may suppose that the $k^{\rm th}$ homotopy group of $U^-$ is non-trivial.   
Thus there a continuous
map $f:\Ss^k\to U^-$ that is homotopically nontrivial in $U^-$.

By dilating, we may assume that $f(\Ss^k)$ lies in $\BB\cap U^-$. 
We will assume from now on that $\eps$ is sufficiently small that $u<-\eps$ on $f(\Ss^k)$.

Let $W(0)=\{x\in \BB: u(x)<\eps\}$.
For $0<t\le\infty$, let
\[
   W(t) = W(0)\cup \left(\bigcup_{T<t} M_\eps(t) \right).
\]
If we think of the $M_\eps(t)$ as moving forward, then $W(t)$ is the portion of $\BB$ that 
lies behind $M_\eps(t)$.  
Note that 
\[
  W(\infty) = \{x\in \BB: u(x)< \eps\}.
\]
Note also that $W(\infty)$ is star-shaped (by~\eqref{homogeneous}), so $f$ is homotopically trivial
in $W(\infty)$. 

Thus $f$ is homotopically trivial in $W(\infty)$ but not in $W(0)$ (since $W(0)\subset U^-$).
By~theorem~\ref{TopologyTheorem} (together with theorem~\ref{ConvexTypeTheorem})
 in the appendix, 
this implies that there is a point $X_\eps=(x_\eps,t_\eps)$ at which the tangent
flow to $\MM_\eps$ is a shrinking $\partial \BB^j\times \RR^{n-k-1}$ for some $j\le k$.
Thus
\[
    \Theta(\MM_\eps, X_\eps) = d_j \ge d_k.
\]
Now exactly as in the proof of theorem~\ref{MainTheorem}, this implies that
\[
    \Theta(C,\OO) \ge \limsup_{\eps\to 0} \Theta(\MM_\eps, X_\eps) \ge d_k.
\]
\end{proof}

\begin{corollary}
Let $C$ be an area-minimizing hypercone in $\RR^n$ with an isolated singularity at the origin.
Suppose that at least one of the components of $\RR^n\setminus C$ is not contractible.
Then
\[
      \Theta(C) \ge d_{[(n-2)/2]}
\]
where $[(n-2)/2]$ is the greatest integer less than or equal to $(n-2)/2$.
\end{corollary}

This improves on the bound $\Theta(C)\ge d_{n-2}$ given by theorem~1$^*$.

\begin{proof}
Let $U^-$ and $U^+$ be the two components of $\RR^n\setminus C$, and let
$V^-$ and $V^+$ be the corresponding components of $(\partial \BB)\setminus C$.
Note that $U^-$ and $U^+$ are homemorphic to $V^-\times \RR$ and $V^+\times\RR$
and thus are homotopy equivalent to $V^-$ and $V^+$.   

We may suppose that $U^-$ is not contractible, so that one of its homotopy groups
is not trivial.  Let $k$ be the smallest integer $\ge 1$ such that $\pi_k(U^-)$ is nontrivial.

If $k=1$, then $\Theta(C,\OO) \ge d_1 \ge d_{[(n-1)/2]}$ and we are done.
Thus we may assume that $k>1$.
In particular, $U^-$ is simply connected, so by the 
Hurewicz Theorem, the homology group $H_k(U^-)$ is nontrivial.
Thus
\begin{align*}
 0
 &\ne H_k(U^-) \\
 &= H_k(V^-)   \\
 &= H_k(\overline{V^-}) \\
 &= H^{n-k-2}(V^+) \qquad (\text{by Alexander duality}) \\
 &= H^{n-k-2}(U^+) 
\end{align*}
By the universal coefficients theorem, the nontriviality of $H^{n-k-2}(U^+)$
implies that $H_{n-k-3}(U^+)$ and $H_{n-k-2}(U^+)$ cannot both be trivial.
The Hurewicz theorem then implies that  $\pi_j(U^+)$ is nontrivial for some $j\le n-k-2$.

Thus $j+k\le n-2$, so if $p$ is the smaller of $j$ and $k$, then $p\le [(n-2)/2]$.
Hence by theorem~\ref{TopTheorem},
\[
   \Theta(C,\OO) \ge d_p \ge d_{[(n-2)/2]}.
\]
\end{proof}

\section{Examples}\label{Examples}

For positive integers $m$ and $n$, let $C_{m,n}$ be the Simons' cone
\[
    C_{m,n} = \{ (x,y)\in \RR^{m+1}\times\RR^{n+1} = \RR^{m+n+2}:  n|x|^2 = m|y|^2\}.
\]
This cone is minimal, and it is area-minimizing if and only either (i) $m+n\ge 6$ or 
(ii) $m+n=6$ and neither $m$ nor $n$ is equal to $1$ (see~\cite{Lawlor-Criterion}).

The cone divides the unit sphere into two components, one of which is the product of an 
 $m$-sphere and an $n$-ball.
Therefore the corresponding component of the complement of $C$ has nontrivial $m^{\rm th}$ homotopy group, so
according to theorem~\ref{TopTheorem}, 
\[
  \Theta(C_{m,n}) \ge d_m.
\]
Thus the following theorem shows that the constant $d_m$ in theorem~\ref{TopTheorem} cannot be replaced by any larger constant.

\begin{theorem}\label{LimitDensity}
$
  \lim_{n\to\infty} \Theta(C_{m,n}) = d_m.
$
\end{theorem}

\begin{proof} 
The intersection $\Gamma_{m,n}$ of $C_{m,n}$ with the unit sphere is the Cartesian product of the $m$-sphere of radius $\sqrt{\frac{m}{m+n}}$
and the $n$-sphere of radius $\sqrt{\frac{n}{m+n}}$.  

Thus
\begin{align*}
  \area(\Gamma_{m,n}) 
  &= \sigma_m\left(\sqrt{\frac{m}{m+n}}\right)^m \cdot \sigma_n \left( \sqrt{\frac{n}{m+n}} \right)^n
  \\
  &= \sigma_m\sigma_n  \left(\frac{m}{m+n}\right)^{m/2}  \left( \frac{n}{m+n} \right)^{n/2}
 \end{align*}
 where $\sigma_k$ is the area of a $k$-dimensional sphere of radius $k$.  
To get the density of the cone, we divide the area of $\Gamma_{m,n}$ by the area $\sigma_{m+n}$ of a unit
sphere of the same dimension:
\begin{equation}\label{ExactDensity}
\Theta(C_{m,n})
= \
\frac{\sigma_m\sigma_n}{\sigma_{m+n}} \left(\frac{m}{m+n}\right)^{m/2}  \left( \frac{n}{m+n} \right)^{n/2}.
\end{equation}
As usual, we write $A \sim B $ to mean $\lim_{n\to\infty}(A/B)=1$.  Note that
\[
   \left(\frac{m}{m+n}\right)^{m/2} 
   =  \left(\frac{m}n\right)^{m/2} \left(\frac{n}{m+n}\right)^{m/2} \sim \left(\frac{m}n\right)^{m/2}
\]
and
\[
  \left( \frac{n}{m+n} \right)^{n/2} = \left(1+\frac{m}n\right)^{-n/2} \to e^{-m/2},
\]
so
\begin{equation}\label{ok}
\Theta(C_{m,n}) \sim \sigma_m \frac{\sigma_n}{\sigma_{m+n}} \left( \frac{m}{ne}\right)^{m/2}.
\end{equation}

Now
\begin{equation}\label{SphereArea}
  \sigma_k  = (k+1)\omega_{k+1} = (k+1) \frac{\pi\,^{(k+1)/2}}{\Gamma(\frac{k+3}2)},
\end{equation}
where $\omega_{k+1}$ is the volume of the unit ball in $\RR^{k+1}$.

Thus
\begin{equation}\label{SphereRatio}
   \frac{\sigma_n}{\sigma_{m+n}} 
   = \frac{n+1}{m+n+1} \cdot \frac{\omega_{n+1}}{\omega_{m+n+1}}    
   \sim \frac{\omega_{m+1}}{\omega_{m+n+1}} 
   = \frac1{\pi^{m/2}} \cdot \frac{\Gamma(\frac{m+n+3}2)}{ \Gamma(\frac{n+3}2)}.
\end{equation}
From Stirling's approximation 
\[
  \Gamma(z)\sim \sqrt{\frac{2\pi}z}\left( \frac{z}e\right)^z,
\]
one checks that 
\[
   \frac{\Gamma(z+a)}{\Gamma(z)} \sim z^a
\]
as $z\to \infty$ with $a$ fixed.  Thus~\eqref{SphereRatio} becomes
\[
 \frac{\sigma_n}{\sigma_{m+n}} 
 \sim \frac1{\pi^{m/2}} \cdot \left( \frac{n+3}2 \right)^{m/2} 
 \sim  \left( \frac{n}{2\pi} \right)^{m/2}. 
\]
Combining this with~\eqref{ok} gives
\begin{equation}\label{SimonsConeAsymptotics}
\Theta(C_{m,n})
\sim
\sigma_m \left( \frac{m}{2 \pi e} \right)^{m/2},
\end{equation}
which is precisely $d_m$ (see~\eqref{GaussianDensityValues}).
\end{proof}

In a similar manner, one can use~\eqref{ExactDensity} and Stirling's approximation
to check that $\lim_{n\to\infty}\Theta(C_{n,n})=\sqrt2$, which
shows that $\sqrt2$ is the best possible constant in theorem~\ref{MainTheorem}.
Alternatively, one can see that $\sqrt2$ is optimal because
\begin{equation*}
\lim_{n\to\infty} \lim_{m\to \infty} \Theta(C_{m,n})
= 
\lim_{n\to\infty} d_n = \sqrt2.
\end{equation*}
by~\eqref{SimonsConeAsymptotics} and~\eqref{dn-limit}.   

Although the bounds $\sqrt2$ and $d_n$ in theorems~\ref{MainTheorem} 
and~\ref{TopTheorem}
are the best constants independent of dimension, they are not optimal if one  considers cones of a given dimension $N$.
However, the bound from theorem~\ref{TopTheorem} is surprisingly good even when $N$ is small.
(Small here means close to $7$, the smallest dimension $N$ for which there exist $N$-dimensional singular area minimizing hypersones.)  
For example, consider the $8$-dimensional cone $C_{1,6}$, the simplest known area-minimizing hypercone
whose complement contains a component that is not simply connected.
 According to theorem~\ref{TopTheorem}, its density is greater than $d_1=1.520$.
In fact, its density (which can be calculated from~\eqref{ExactDensity} and~\eqref{SphereArea})
 is $1.523$.  Thus the density is only $0.2\%$ higher than the lower bound from theorem~\ref{TopTheorem}.  By contrast, its density is about $8\%$ higher than 
 $\sqrt2=1.414$, the lower bound from 
 theorem~\ref{MainTheorem}.

\section{Open Problems}

\begin{enumerate}
\item Prove that in Theorem~\ref{MainTheorem}, we can drop one or more of the
following hypotheses: that $C$ be minimizing, that $C$ be topologically nontrivial, and that
$C$ have an isolated singularity.
\item Prove that in Theorem~\ref{TopTheorem}, we can
drop one or both of the hypotheses that $C$ be area-minimizing and that $C$ have
an isolated singularity.
\item (conjectured by Solomon.) For $m\ge 1$, prove that the Simon's cone $C_{m,m}$ is the $(2m+1)$-dimensional minimal 
 (or area-minimizing) 
 hypercone of least possible density.  For $m=1$, one has to exclude cones with soap-film-like triple
 junctions, since the density of $3$ half planes meeting along a common edge is $3/2$,
 which is less than $\Theta(C_{1,1})$.  However, in higher dimensions this exclusion
 is not necessary since $\Theta(C_{m,m})<3/2$ for $m>1$.
 \item  (Conjectured by Solomon.)  
 Prove that the cone $C_{m,m+1}$ is the $(2m+2)$-dimensional minimal (or area-minimizing)
 hypercone of least possible density.
 \item Prove that $C_{m,n}$ has  the least possible density among all $(m+n+1)$-dimensional 
 minimal (or area-minimizing) hypercones $C$ such that at least one of the components of the complement
 has nontrivial $m^{\rm th}$ homotopy group.
 \item Prove lower density bounds for minimal or for area-minimizing cones of codimension $>1$.
 \item Find a conceptual explanation for why the bounds in this paper are sharp.  The authors find
 it mysterious that the method of proof here gives such good bounds.
 \end{enumerate}
 
\appendix

\section{mean curvature flow}\label{appendix}

In this section, we state the fundamental results about  existence and 
singularity structure
  for  a mean-convex hypersurface $M_t$ moving by mean
 curvature in an ambient space $N$, where the prescribed Dirichlet data
 consist of the initial surface $\Sigma=M_0$ and the motion 
 $t\mapsto \Gamma_t = \partial M_t$ of the boundary.
Throughout this section, we make the following assumptions:

\begin{enumerate}[\upshape(1)]
\item $N$ is a smooth Riemannian $n$-manifold with smooth boundary.
   At each boundary point, the mean curvature is a nonnegative multiple of the inward
   unit normal.  (In this paper, $N$ is the unit ball in $\RR^n$.)
\item $t\in[0,\infty)\mapsto \Gamma_t$ is a smooth, $1$-parameter family of smoothly embedded,
closed $(n-2)$-manifolds in $\partial N$.   As $t\to\infty$, the $\Gamma_t$ converge
smoothly to a smoothly embedded $(n-2)$ manifold $\Gamma_\infty$.
\item\label{monotonic} 
        The family $t\mapsto \Gamma_t$ is monotonic is the following sense:
        $\Gamma_t$ is boundary of a region $U_t\subset \partial N$ where
         $U(t)\supset U(t')$ for $t\le t'$.      
\item The initial surface
   $\Sigma$ is a smoothly embedded, compact $(n-1)$-manifold in $N$ such that
\[
  \partial \Sigma = \Gamma_0
\]
and such that $\Sigma$ and $\partial N$ are nowhere tangent\footnote{Actually, 
it suffices to assume that $\Sigma$ and $\overline{U_0}$ are nowhere tangent.
Indeed, the hypothesis is stated that way in  
  \cite{White-Subsequent}. This allows, for example, the initial surface
  $\Sigma$ to be $(\partial N)\setminus U_0$.}.
\item The surfaces $\Sigma$ and $U_0$ together bound a region $\Omega$ in $N$,
 and at each point of $\Sigma$, the mean curvature vector
 is a nonnegative multiple of
 the unit normal pointing into $\Omega$.
\item\label{moving}
         If $\Sigma$ is a minimal surface, then $\Gamma_t\ne \Gamma_0$
         for $t>0$. (This is to guarantee that surface $M_t$ starts moving as soon as $t$ is positive.)
\end{enumerate}

\begin{theorem}\label{BoundaryExistenceTheorem}
Under the hypotheses above,
 there is a unique weak solution $t\in[0,\infty)\mapsto M_t$ of mean curvature flow
such that $M_0=\Sigma$ and  such that $\partial M_t=\Gamma_t$ for all $t$. 
The surfaces $M_t\setminus \Gamma_t$ are disjoint (for distinct values of $t$), and the 
time-of-arrival function
\begin{align*}
&\tau: \overline{\Omega}\setminus \partial N \to [0,\infty] \\
&\tau(x) = \begin{cases}
t &\text{if $x\in M_t$}\\
\infty &\text{if $x\notin \cup_{0\le t < \infty}M_t$}
\end{cases}
\end{align*}
is a continuous function.  Each $M_t$ is rectifiable, and the multiplicity-one varifolds
associated to the $M_t$ form a Brakke flow.

Furthermore, there is a compact subset (the singular set) $Q$ of $\Omega$ with the following properties:
\begin{enumerate}[\upshape(1)]
\item The set $Q$ has Hausdorff dimension at most $(n-2)$, and the spacetime singular set
has parabolic Hausdorff dimension at most $(n-2)$.
\item\label{Each-M-assertion}
 Each $M_t\setminus Q$ with $t\in(0,\infty)$ is a smooth, properly embedded submanifold of 
    $N\setminus Q$ with
boundary $\Gamma_t$, and the mean curvature of $M_t\setminus(Q\cup \Gamma_t)$ is 
 everywhere positive.
\item If  $t>0$ and $t(i)\to t$, then $M_{t(i)}$ converges smoothly to $M_t$ away from $M_t\cap Q$.
\item If $t(i)\to 0$, then $M_{t(i)}$ converges in $C^{1,\alpha}$ (for every $\alpha\in (0,1)$)
 to $M_0$, and the convergence is smooth except at the boundary.
 \item\label{infinity-assertion}
  The surface $M_t$ converges as $t\to\infty$ to a minimal variety $M_\infty$, 
   and the convergence is smooth away from the singular set $Q\cap M_\infty$, which
   has Hausdorff dimension at most $n-8$.
\end{enumerate}
\end{theorem}

See~\cite{White-Subsequent}*{theorem~4} for the proof.

Of course ``there is a unique weak solution'' is somewhat informal.
However, uniqueness is not needed in this paper.
(The precise uniqueness statement is: the flow $t\in[0,\infty)\mapsto M_t$ is the level set
flow generated by $\Sigma$.  Note that this flow is non-fattening 
since by assertion~\ref{Each-M-assertion}
of the theorem, $M_t$ contains no non-empty open subset of $N$.
  See~\cite{White-Topology} for level set
flow of surfaces with boundary.)

If $M$ is an $(n-1)$-dimensional hypersurface and $x\in M$ is a point where the mean curvature is nonzero,
let 
\[
    \Phi(M,x) = \frac{\kappa_1}h
\]
where 
$
     \kappa_1\le \kappa_2 \le \dots \le \kappa_{n-1}
$
 are the principal
curvatures of $M$ at $x$ and 
\[
     h=\sum_i\kappa_i>0
\]
 is the mean curvature of $M$ at $x$.
(One can think of $\Phi(M,x)$ as a dimensionless measure of convexity.
In particular, $\Phi(M,x)\ge 0$ if and only if $M$ is convex to second order at $x$.) 
We say a singularity $(x,t)$ with $x\in M_t$ has {\em convex type}
provided:
\begin{enumerate}
\item\label{cylinder} Each tangent flow at $(x,t)$ is a 
self-similarly shrinking $\Ss^k\times \RR^{n-k}$
  for some $k\ge 1$.
\item\label{umbilicity} If $x_i\in M_{t_i}$ is a sequence of regular points converging to $(x,t)$, then
\[
  \liminf_{i\to\infty}\Phi(M_{t_i},x_i) \ge 0.
\]
\end{enumerate}

It seems likely that all of the finite-time singularities of the flow in 
theorem~\ref{BoundaryExistenceTheorem}
 must have convex
type.  In some situations, it is known that the singularities have convex type:

\begin{theorem}\label{ConvexTypeTheorem}
Let $t\mapsto M_t$ be as in 
theorem~\ref{BoundaryExistenceTheorem}.
Suppose that either
\begin{enumerate}
\item  $n<8$,  or 
\item $N\subset \RR^n$ (with the Euclidean metric) and $M_\infty$ is smooth.
\end{enumerate}
Then all of the finite-time singularities of the flow have convex type.
\end{theorem}

See~\cite{White-Subsequent}*{theorems 5 and 6} for the proof.
(It may seem peculiar to 
require any hypothesis on $M_\infty$ for a parabolic problem.
The hypothesis arises because the proof is by elliptic regularization.)

Now extend the time-of-arrival function $\tau$ in 
theorem~\ref{BoundaryExistenceTheorem} to all of $N\setminus \partial N$
by setting $\tau(x)=-\infty$ for $x \notin\overline{\Omega}$.
Of course $\tau$ will now be discontinuous along $M_0$.
Let $W(t)=\{x: \tau(x)<t\}$.  If one thinks of $M_t$ as moving forward, then for $t\ge 0$, 
the region $W(t)$ is the region in the interior of $N$ that lies behind the surface $M_t$.
Note the the topology of $W(t)$ can change only when there are singularities in the flow.
Not surprisingly, the way that the topology changes gives information about those singularities:

\begin{theorem}\label{TopologyTheorem}
Let $t\mapsto M_t$ be as in theorem~\ref{BoundaryExistenceTheorem}, 
and suppose that all the finite-time singularities have convex type.
Let $0\le t<t'\le\infty$ and suppose that $f:\Ss^k\to W(t)$ is homotopically trivial  
in $W(t')$ but not in $W(t)$.  Then there is a singularity $(x_0,t_0)$
with $t\le t_0<t'$ that has as a tangent flow a shrinking $\Ss^j\times\RR^{n-j}$
for some $j\le k$.
\end{theorem}

See~\cite{White-Changes}*{theorem~5.3} for the proof.
(The theorem there is stated under the 
additional hypotheses in theorem~\ref{ConvexTypeTheorem}
above, but as explained in the paragraph following 
  \cite{White-Changes}*{theorem~5.3}, those hypotheses are used only
  to guarantee that the finite-time singularities have convex type.)

The Gaussian density $d_k$ of a shrinking $\Ss^k\times\RR^{n-k}$ 
 is  a strictly decreasing function of $k$. 
(See~\eqref{dn-limit} in section~\ref{Proof-One-Section}.)
Thus theorem~\ref{TopologyTheorem} implies the following:
to kill an element of the $k^{\rm th}$ homotopy of $W(t)$ requires a singularity
of Gaussian density $\ge d_k$.

\end{document}